\documentclass[12pt,a4paper]{amsart}
\usepackage{geometry}
\usepackage{amsmath,amssymb,amsfonts,amscd,amsthm,wasysym,color,enumitem,indentfirst,graphicx,booktabs}
\usepackage[bookmarksnumbered,colorlinks,linktocpage,plainpages]{hyperref}
\hypersetup{pdfstartview={FitH}}
\usepackage{mathtools}

\numberwithin{equation}{section}
\newtheorem{theorem}{Theorem}[section]

\newtheorem{lemma}[theorem]{Lemma}
\newtheorem{proposition}[theorem]{Proposition}

\theoremstyle{definition}
\newtheorem{definition}[theorem]{Definition}

\newtheorem{remark}[theorem]{Remark}

\usepackage{tikz}

\allowdisplaybreaks[4]

\makeatletter
\@namedef{subjclassname@2020}{\textup{2020} Mathematics Subject Classification}
\makeatother

\def\abs#1{\left|#1\right|}
\date{}

\def\R{\mathbb{R}}
\def\C{\mathbb{C}}
\def\H{\mathbb{H}}
\def\O{\mathbb{O}}

\begin{document}
\title[Octonionic Witt basis]{Explicit   Witt basis over the tesor product of  Clifford algebras and   octonions}

\author{Yong Li}
\address{School of Mathematics and Statistics, Anhui Normal University, Wuhu 241002, Anhui, People's Republic of China.}
\email{leeey@ahnu.edu.cn}

\author{Guangbin Ren}
\address{%
 School of Mathematical Sciences, University of Science and Technology of China,Hefei, Anhui 230026, People's Republic of China.}
\email{rengb@ustc.edu.cn}

\author{Haiyan Wang}
\address{School of Science, Tianjin University of Technology and Education,
300222, Tianjin,   People's Republic of China.
}
\email{whaiyan@mail.ustc.edu.cn}

\thanks{Yong Li is supported by University Annual Scientific Research Plan of Anhui Province(2022AH050175).
Guangbin Ren is supported  by the National Natural Science Foundation of China (Grant Nos.12171448)}


\begin{abstract}
In this article, we investigate how the Witt basis serves as a link between real and complex variables in higher-dimensional spaces. Our focus is on the detailed construction of the Witt basis within the tensor product space combining Clifford algebra and multiple octonionic spaces. This construction effectively introduces complex coordinates. The technique is based on a specific subgroup of octonionic automorphisms, distinguished by binary codes. 
This method allows us to perform  a Hermitian analysis of the complex structures within the tensor product space. 
  \end{abstract}
\maketitle
\tableofcontents
\section{Introduction}
Clifford analysis represents an advanced generalization of the theory of a single complex variable into higher dimensions. It has transitioned into a complex analog, known as Hermitian Clifford analysis, through the utilization of the Witt basis, a pivotal development cited in various works \cite{FsRV,BDS,BJFVS,SF,SAS}.

The concept of the Witt basis within Clifford algebra was first introduced by Sabadini and Sommen in 2002 \cite{SF}, marking the inception of Hermitian Clifford analysis. This field saw substantial development in 2007 through the contributions of Brackx, Bur\v{e}s, De Schepper, Eelbode, Sommen, and Sou\v{c}ek \cite{BJFVS}, which included its extension into the realm of quaternionic analysis, thereby initiating the study of quaternionic Hermitian analysis \cite{PIS}.

Given the advancements in Hermitian Clifford analysis within the contexts of complex numbers and quaternions, extending this study to the octonionic setting emerges as a logical progression. The octonion algebra $\mathbb{O}$, characterized by its non-associative and non-commutative nature and its automorphism group being the exceptional simple Lie group $G_2$, finds significant applications in various theories including string theory, the special theory of relativity, and quantum theory. The model of the universe in M-theory, for instance, is conceptualized as the Minkowski space $\mathbb{R}^{1,3}$ times a $G_2$ manifold of very small diameter \cite{BOR, OK}. The rich function theory of octonionic algebra is well-documented in references \cite{WR,JR,YR}.

In quaternionic Hermitian analysis, the quaternionic Witt basis of the tensor product of several quaternionic variables with Clifford algebras $\H^n \otimes_{\R}C\ell_{4n}\ $ is fundamental. This article aims to establish the octonionic Witt basis on the tensor product of several octonionic variables with Clifford algebras, setting the groundwork for octonionic Hermitian analysis.

The octonionic Witt basis is instrumental in generating Hermitian variables, which can be viewed as complex versions of twistor vectors. These twistor vectors serve as an orthonormal basis extension of $\mathbb{R}^8$ starting from a generic unit vector $X$ in $\mathbb{R}^8$. The study of the octonionic Witt basis reveals its deep connection with binary expansions. Specifically, certain involutions derived from the binary expansions of $\mathbb{Z}_8$ form a subgroup of $G_2$, crucial for constructing the octonionic Witt basis of $\O^n\otimes Cl_{8n}$. These involutions allow for the expression of projections in terms of the finite $G_2$ subgroup, marking a cornerstone in the theory of octonionic Hermitian analysis.

Furthermore, the octonionic Witt basis facilitates the introduction of Hermitian vector derivatives. Recalling the construction from the quaternionic scenario, in the tensor product $\H^n \otimes_{\R}C\ell_{4n}$, Hermitian vector derivatives are constructed involving the quaternionic basis and the Cauchy-Fueter operator. Similarly, with the octonionic basis and the octonionic Dirac operator, we introduce the octonionic Hermitian Dirac operators as the primary focus of octonionic Hermitian Clifford analysis, thus broadening the scope and depth of this field.

\section{Preliminaries}
This section introduces the fundamental concepts of octonion and Clifford algebras.

\subsection{Octonion Algebra}
The octonion algebra $\mathbb{O}$ is defined as a non-associative, non-commutative, normed division algebra over the real numbers $\mathbb{R}$. We designate $\mathbf{e_1}, \ldots, \mathbf{e_7}$ as its natural basis, satisfying the relation
$$\mathbf{e_i}\mathbf{e_j} + \mathbf{e_j}\mathbf{e_i} = -2\delta_{ij}$$
for all $i, j = 1, \ldots, 7$. The unit element is denoted by $\mathbf{e_0} = 1$. The multiplication rules of the octonions are encapsulated in the Fano plane, where each vertex corresponds to one of the basis elements $\mathbf{e_1}, \ldots, \mathbf{e_7}$, and every oriented line represents a quaternionic triple. The multiplication of any two basis elements is determined by the third element on their connecting line, with orientation influencing the sign (see \cite{Baez} for the Fano plane).

\begin{figure}[ht]
\centering
  \includegraphics[width=6cm]{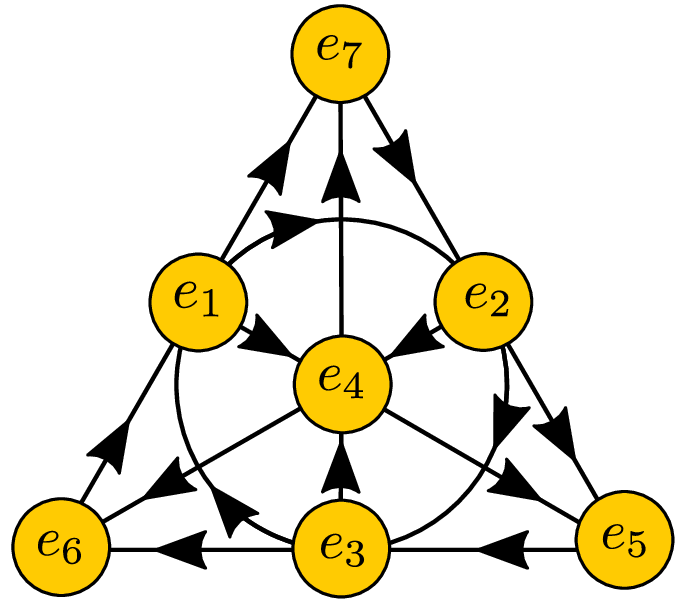}
  \caption{Fano plane}
  \label{fig:1}
\end{figure}

The automorphism group of the octonion algebra, $G_2$, is a 14-dimensional Lie group, defined as\cite{SG}
$$G_2 = \text{Aut}(\mathbb{O}) = \{g \in GL(8, \mathbb{R}) : g^*\phi = \phi, g\mathbf{e_0} = \mathbf{e_0}\},$$
where $\phi$ is a specific 3-form defined in terms of the dual basis $\{\mathbf{e^j}\}_{j=0}^7$ by
$$\phi = \mathbf{e}^{123} + \mathbf{e}^{145} + \mathbf{e}^{176} + \mathbf{e}^{257} + \mathbf{e}^{246} + \mathbf{e}^{347} + \mathbf{e}^{365},$$
with $\mathbf{e}^{ijk} = \mathbf{e}^i \wedge \mathbf{e}^j \wedge \mathbf{e}^k$.

\subsection{Clifford Algebra}
In $\mathbb{R}^8$, the canonical orthonormal basis is $\{g_0, g_1, \ldots, g_7\}$. The Clifford algebra $C\ell_8$, an associative but non-commutative algebra, is generated by the unit 1 and the basis elements $g_i$, obeying the relations
$$g_ig_j + g_jg_i = -2\delta_{ij}.$$
Thus, $\mathbb{R}^8$ is considered a subspace of $C\ell_8$.

\subsection{Isometry between $\mathbb{O}$ and $\mathbb{R}^8$}
An isometry between the octonions $\mathbb{O}$ and the Euclidean space $\mathbb{R}^8$ is established through the linear isomorphism $\Phi: \mathbb{R}^8 \rightarrow \mathbb{O}$, defined by
$$\Phi\left(\sum_{i=0}^7 x_i g_i\right) = \sum_{i=0}^7 x_i \mathbf{e_i}.$$
The inner product in both $\mathbb{R}^8$ and $\mathbb{O}$, denoted by $(\cdot, \cdot)$, is defined in $\mathbb{O}$ as
$$(p, q) = \text{Re}(\overline{p} q), \quad \forall p, q \in \mathbb{O},$$
making the map $\Phi$ an isometry between these two Hilbert spaces.

\section{Finite $G_2$-subgroup}

The construction of the Witt basis is reliant on a finite subgroup of automorphisms of the octonion algebra, composed of eight involutions induced by binary expressions. We define the set of integers modulo 8 as
\begin{equation}
	\mathbb{Z}_8 := \{0, 1, \ldots, 7\}.
\end{equation}
For any integer $i \in \mathbb{Z}_8$, its binary expression is given by
\begin{equation}
	i = i_3 2^2 + i_2 2 + i_1,
\end{equation}
where $i_1, i_2, i_3 \in \{0,1\}$. This binary expansion allows for a canonical identification $\mathbb{Z}^8 \cong (\mathbb{Z}_2)^3$.

Utilizing the binary expression, we introduce an automorphism of the octonions, $J_i: \mathbb{O} \rightarrow \mathbb{O}$. Considering $\mathbf{e_0}, \mathbf{e_1}, \ldots, \mathbf{e_7}$ as the standard basis of $\mathbb{O}$, we define
\begin{align*}
	J_i(\mathbf{e_1}) &= (-1)^{i_1}\mathbf{e_1}, \\
	J_i(\mathbf{e_2}) &= (-1)^{i_2}\mathbf{e_2}, \\
	J_i(\mathbf{e_4}) &= (-1)^{i_3}\mathbf{e_4},
\end{align*}
extending this definition to the entirety of $\mathbb{O}$ while respecting its multiplication structure. This approach is valid as $\{\mathbf{e_1}, \mathbf{e_2}, \mathbf{e_4}\}$ generate the algebra $\mathbb{O}$.

It is evident that $J_0$ acts as the identity map, with the remaining maps being involutions of the octonions.

Recalling that $G_2 = \text{Aut}(\mathbb{O})$, it includes $(\mathbb{Z}_2)^3$ as a finite subgroup through the natural embedding
\begin{equation}
	(\mathbb{Z}_2)^3 \hookrightarrow G_2,
\end{equation}
originating from the subgroup $\{J_0, J_1, \ldots, J_7\} \subset G_2$.

This embedding is facilitated by the map
\begin{equation}
	J: (\mathbb{Z}_2)^3 \rightarrow G_2,
\end{equation}
where each element $(i_3, i_2, i_1)$ is mapped to $J_i$ according to its binary expression.

The set $\{J_0, J_1, \ldots, J_7\}$ plays a crucial role in constructing the Witt basis. Utilizing the group $\{J_i\}_{i=0}^7$, we define the orthogonal projection
\begin{equation}
	\mathbb{P}_i: \mathbb{O} \rightarrow \mathbb{R}\mathbf{e_i},
\end{equation}
for each $i = 0, 1, \ldots, 7$. To achieve this, we introduce auxiliary maps since, by definition,
\begin{equation}
	J_j(\overline{\mathbf{e_i}}) = \pm \mathbf{e_i}.
\end{equation}
We denote these signs by $\sigma(j,i) = 0$ or $1$, such that
\begin{equation}
	J_j(\overline{\mathbf{e_i}}) = (-1)^{\sigma(j,i)} \mathbf{e_i},
\end{equation}
for all $i, j = 0, \ldots, 7$.
\begin{lemma}\label{lem-inv-235} If  $x=\sum_{i=0}^7x_i\mathbf{e_i}\in \O,$   then
\begin{eqnarray} \label{real} x_i=  \frac{1}{8}\mathbf{e_i}\Big(\sum_{j=0}^7(-1)^{\sigma(j,i)} J_j(x)\Big) \end{eqnarray}
for all $i=0, 1,\ldots, 7$.
\end{lemma}

\begin{proof}
For any octonion
 $$x=\sum_{i=0}^7x_i\mathbf{e_i}\in \O,$$   its real part can be expressed as
$$\mbox{Re}\ x= x_0=\frac{1}{8}\sum_{i=0}^7J_i(x).$$

To prove   \eqref{real}, direct  calculation  shows
\begin{eqnarray*}
  x_i&=&Re(\overline{\mathbf{e_i}}x)\notag
  \\ &=&\frac{1}{8}\sum_{j=0}^7J_j(\overline{\mathbf{e_i}}x) \notag
  \\ &=&\frac{1}{8}\sum_{j=0}^7J_j(\overline{\mathbf{e_i}})J_j(x)  \\
  &=&\frac{1}{8}\sum_{j=0}^7(-1)^{\sigma(j,i)} \mathbf{e_i}J_j(x)\notag
  \\ &=& \frac{1}{8}\mathbf{e_i}(\sum_{j=0}^7(-1)^{\sigma(j,i)} J_j(x)).
\end{eqnarray*}
\end{proof}

We can restate Lemma \ref{lem-inv-235} in the formalism of operators.

\begin{lemma} For each  $i=0, 1, \ldots, 7$,     the orthogonal projection
$$\mathbb P_i: \mathbb O\longrightarrow \mathbb R \mathbf{e_i}$$
 is given by
 \begin{eqnarray} \label{real-039} \mathbb P_i=  \frac{1}{8}\mathbf{e_i} \sum_{j=0}^7(-1)^{\sigma(j,i)} J_j.   \end{eqnarray}
\end{lemma}
\section{Octonionic Witt Basis}

We explore the tensor product of two algebras, $\O \otimes C\ell_8$, to introduce a Witt basis for this vector space. For simplicity, the tensor product symbol will be omitted; for instance, we write $\mathbf{e_i}g_A$ to denote $\mathbf{e_i} \otimes g_A$.

The automorphism $J_j \in \text{Aut}\O$ for $j = 0, 1, \ldots, 7$ has been previously introduced. We now extend $J_j$ to an automorphism in $\O \otimes C\ell_8$ as follows:
$$J_j(\mathbf{e_i}g_A) := J_j(\mathbf{e_i})g_A.$$

\begin{lemma}
	For any $j = 0, 1, \ldots, 7$, $J_j \in \text{Aut}(\O \otimes C\ell_8)$.
\end{lemma}

\begin{proof}
	Direct verification shows that
	$$J_j(\mathbf{e_i}g_A \mathbf{e_k}g_B) = J_j(\mathbf{e_i e_k}g_Ag_B) = J_j(\mathbf{e_i}g_A)J_j(\mathbf{e_k}g_B).$$
\end{proof}

In $\O$, the conjugate of an element $x = x_0 + \sum_{i=1}^7 x_i\mathbf{e_i}, x_i \in \R$, is defined by
$$\overline{x} = x_0 - \sum_{i=1}^7 x_i\mathbf{e_i}.$$

We set
\begin{equation}\label{basis-093}
 \Omega := \sum_{i=0}^7 \overline{\mathbf{e_i}}g_i \in \O \otimes C\ell_8,
\end{equation}
where $\{\mathbf{e_0}, \mathbf{e_1}, \ldots, \mathbf{e_7}\}$ is the natural basis of $\O$ and $\{g_0, g_1, \ldots, g_7\}$ is the standard basis of $\R^8$, which also generates the Clifford algebra $C\ell_8$.

\begin{definition}
	The octonionic Witt basis of $\O \otimes C\ell_8$ is given by
	\begin{equation}
		\begin{dcases}
			f_0 &= J_0(\Omega), \\
			f_1 &= J_1(\Omega), \\
			&\vdots \\
			f_7 &= J_7(\Omega),
		\end{dcases}
	\end{equation}
	where $\Omega$ is as defined above.
\end{definition}

\begin{remark} Here we present the explicit expression of the Witt basis $\{f_j\}_{j=0}^7$ for octonions:
\begin{equation*}
    \begin{dcases}
    f_0  =\mathbf{e_0}g_0-\mathbf{e_1}g_1-\mathbf{e_2}g_2-\mathbf{e_3}g_3-\mathbf{e_4}g_4-\mathbf{e_5}g_5-\mathbf{e_6}g_6-\mathbf{e_7}g_7; \\
    f_1  =\mathbf{e_0}g_0+\mathbf{e_1}g_1-\mathbf{e_2}g_2+\mathbf{e_3}g_3-\mathbf{e_4}g_4+\mathbf{e_5}g_5-\mathbf{e_6}g_6+\mathbf{e_7}g_7; \\
    f_2  =\mathbf{e_0}g_0-\mathbf{e_1}g_1+\mathbf{e_2}g_2+\mathbf{e_3}g_3-\mathbf{e_4}g_4-\mathbf{e_5}g_5+\mathbf{e_6}g_6+\mathbf{e_7}g_7; \\
    f_3  =\mathbf{e_0}g_0+\mathbf{e_1}g_1+\mathbf{e_2}g_2-\mathbf{e_3}g_3-\mathbf{e_4}g_4+\mathbf{e_5}g_5+\mathbf{e_6}g_6-\mathbf{e_7}g_7; \\
    f_4  =\mathbf{e_0}g_0-\mathbf{e_1}g_1-\mathbf{e_2}g_2-\mathbf{e_3}g_3+\mathbf{e_4}g_4+\mathbf{e_5}g_5+\mathbf{e_6}g_6+\mathbf{e_7}g_7; \\
    f_5  =\mathbf{e_0}g_0+\mathbf{e_1}g_1-\mathbf{e_2}g_2+\mathbf{e_3}g_3+\mathbf{e_4}g_4-\mathbf{e_5}g_5+\mathbf{e_6}g_6-\mathbf{e_7}g_7; \\
    f_6  =\mathbf{e_0}g_0-\mathbf{e_1}g_1+\mathbf{e_2}g_2+\mathbf{e_3}g_3+\mathbf{e_4}g_4+\mathbf{e_5}g_5-\mathbf{e_6}g_6-\mathbf{e_7}g_7; \\
    f_7  =\mathbf{e_0}g_0+\mathbf{e_1}g_1+\mathbf{e_2}g_2-\mathbf{e_3}g_3+\mathbf{e_4}g_4-\mathbf{e_5}g_5-\mathbf{e_6}g_6+\mathbf{e_7}g_7.
      \end{dcases}
  \end{equation*}
\end{remark}
Analogous to Lemma \ref{lem-inv-235}, we have
\begin{proposition}\label{lem-inv-236}
	Assume $$p = \sum_{i=0}^7 \mathbf{e_i}p_i \in \O \otimes_{\R}C\ell_8$$ with $p_i \in C\ell_8$. Then
	$$p_i = \frac{1}{8}\mathbf{e_i}\left(\sum_{j=0}^7 (-1)^{\sigma(j,i)} J_j(p)\right)$$
	for all $i = 0, 1, \ldots, 7$.
\end{proposition}

\begin{proof} For any
$$x=\sum_{i=0}^7x_i\mathbf{e_i}\in \O,$$
we have expressed its coefficients in  \eqref{real}  as
$$
 x_i=  \frac{1}{8}\mathbf{e_i}\Big(\sum_{j=0}^7(-1)^{\sigma(j,i)} J_j(x)\Big)
 $$
 for any  $i=0,1,\ldots, 7$.

By check carefully its proof, the result  can be extended  to $\O\otimes C\ell_8$, where
we regarded  $J_i\in\mbox{Aut}(\O\otimes C\ell_8)$.

By definition,
 \begin{equation*}\label{basis-884} p=\sum_{i=0}^7\overline{\mathbf{e_i}}p_i \in \O\otimes C\ell_8 \end{equation*}
so that its coefficients $g_i$ can be expressed as
 \begin{equation*}\label{basis}
  p_i=\frac{1}{8}\mathbf{e_i}\Big(\sum_{j=0}^7(-1)^{\sigma(j,i)} J_j(p)\Big).
\end{equation*}
\end{proof}

We now prove that any basis in $\mathbb{R}^8$ can be expressed linearly by the Witt basis of octonions, with coefficients being octonionic numbers.
\begin{lemma}
	The basis $\{g_j\}_{j=0}^7$ of $\R^8$ can be $\O$-linearly expressed via the Witt basis $\{f_j\}_{j=0}^7$ as
	$$g_i = \frac{1}{8}\mathbf{e_i}\left(\sum_{j=0}^7 (-1)^{\sigma(j,i)} f_j\right).$$
\end{lemma}

\begin{proof}
By definition, we have
 \begin{equation*}\label{basis-884} \Omega=\sum_{i=0}^7\overline{\mathbf{e_i}}g_i \in \O\otimes C\ell_8 \end{equation*}
so that we can apply Proposition \ref{lem-inv-236} to express its coefficients $g_i$ as
 \begin{equation*}\label{basis}
  g_i=\frac{1}{8}\mathbf{e_i}\Big(\sum_{j=0}^7(-1)^{\sigma(j,i)} J_j(\Omega)\Big)=\frac{1}{8}\mathbf{e_i}\Big(\sum_{j=0}^7(-1)^{\sigma(j,i)} f_j\Big).
\end{equation*}
\end{proof}

This Witt basis is crucial for constructing Hermitian variables, further enriching the algebraic structure and facilitating the exploration of octonionic spaces and their applications.

Because the Witt basis $\{f_j\}_{j=0}^7$ of octonions can be linearly expressed with respect to the standard orthogonal basis of
$\R^8$ in terms of octonions.
For each vector
\begin{equation}\label{eq:XXX}
	X=\sum_{i=0}^7x_ig_i \in \R^8,
\end{equation}
its coefficient
$Z_i$ under the octonion Witt basis is uniquely determined.
This coefficient is the important Hermitian variable in octonionic Hermitian analysis.

\begin{definition}\label{Hvandtv}
Consider the vector  in $\mathbb R^8$
  $$X=\sum_{i=0}^7x_ig_i \in \R^8, $$
  where $\{g_0, g_1,\ldots g_7 \}$ is a basis of
$\R^8$.
We recall the definition of $f$ and $\Omega$ in   \eqref{basis-884} and
 \eqref{basis-093}.

We define the {\it Hermitian}  variables in terms of the Witt basis  $\{f_0, f_1,\ldots,  f_7 \}$ in  $\O\otimes C\ell_8$ as
\begin{equation}\label{eq:basis-027}
    \begin{dcases}
    Z_0  :=f_0 J_0   (\Phi(X));  \\
     Z_1 :=f_1 J_1   (\Phi(X));   \\
     \quad  \cdots\\
    Z_7  :=f_7 J_7  (\Phi(X)).
    \end{dcases}
  \end{equation}
Namely,
 \begin{equation}\label{basis-897} Z_i:=J_i(\Omega\Phi(X))=J_i(\Omega) J_i(\Phi(X))=f_iJ_i(\Phi(X))\in  \O\otimes C\ell_8
  \end{equation} for any $i=0,\ldots, 7$.
\end{definition}

\begin{definition}
Consider the vector  in $\mathbb R^8$
  $$X=\sum_{i=0}^7x_ig_i \in \R^8, $$
  where $\{g_0, g_1,\ldots g_7 \}$ is a basis of
$\R^8$.
 The {\it twistor}  vectors  $X_i\in \R^8$ are  defined by
 \begin{equation}\label{eq:basis-532}
    \begin{dcases}
    X_0  :=\Phi^{-1}(\Phi(X)\overline{\mathbf{e_0}});   \\
    X_1  :=\Phi^{-1}(\Phi(X)\overline{\mathbf{e_1}});   \\
     \quad  \cdots\\
     X_7  :=\Phi^{-1}(\Phi(X)\overline{\mathbf{e_7}}).
     \end{dcases}
  \end{equation}
Namely, $$X_i  :=\Phi^{-1}(\Phi(X)\overline{\mathbf{e_i}})\in \R^8,$$for any $i=0,\ldots, 7$.
\end{definition}
\begin{remark}
 Here we present the explicit expression of the {\it twistor}  vectors  $X_i\in \R^8$ :
	\begin{equation*}
    \begin{dcases}
    X_0  :=x_0g_0+x_1g_1+x_2g_2+x_3g_3+x_4g_4+x_5g_5+x_6g_6+x_7g_7;   \\
     X_1  :=x_1g_0-x_0g_1-x_3g_2+x_2g_3-x_5g_4+x_4g_5+x_7g_6-x_6g_7;   \\
      X_2  :=x_2g_0+x_3g_1-x_0g_2-x_1g_3-x_6g_4-x_7g_5+x_4g_6+x_5g_7;   \\
       X_3  :=x_3g_0-x_2g_1+x_1g_2-x_0g_3-x_7g_4+x_6g_5-x_5g_6+x_4g_7;   \\
        X_4  :=x_4g_0+x_5g_1+x_6g_2+x_7g_3-x_0g_4-x_1g_5-x_2g_6-x_3g_7;   \\
         X_5  :=x_5g_0-x_4g_1+x_7g_2-x_6g_3+x_1g_4-x_0g_5+x_3g_6-x_2g_7;   \\
          X_6  :=x_6g_0-x_7g_1-x_4g_2+x_5g_3-x_2g_4-x_3g_5-x_0g_6+x_1g_7;   \\
           X_7  :=x_7g_0+x_6g_1-x_5g_2-x_4g_3+x_3g_4+x_2g_5-x_1g_6-x_0g_7.
     \end{dcases}
  \end{equation*}
\end{remark}

The relationship between the Hermitian variables
 $Z_i\in \O\otimes C\ell_8$ and the  {\it twistor} vectors
$X_i\in \R^8$, is akin to the relationship considered in the case of
$\R^2$, where we consider the relationship between
$z,\bar{z}$ and the system of $x,y$. The specific relationship is given by the following lemma.

\begin{lemma}\label{prop3.5} For any $X\in\mathbb R^8$, we have
\begin{equation*}
  Z_0= \Omega \Phi(X)=\sum_{j=0}^7\mathbf{e_j}X_j.
\end{equation*}

\end{lemma}
\begin{proof}
  Suppose that
 \begin{equation}\label{eq:basis-589}
  \Omega \Phi(X)=\Big( \sum_{k=0}^7\overline{\mathbf{e_k}}g_k \Big)\Big(\sum_{j=0}^7x_j\mathbf{e_j}\Big)=\sum_{i=0}^7 \mathbf{e_i}Y_i,
\end{equation}
where $$Y_i\in \R^8\subset C\ell_8.$$
We only need to prove that $$\Phi(Y_i)=\Phi(X)\overline{\mathbf{e_i}}.$$

One can easily check from \eqref{eq:basis-589}  that $$Y_i=\sum_{\overline{\mathbf{e_k}}\mathbf{e_j}=\pm \mathbf{e_i}} g_kx_j(\overline{\mathbf{e_k}}\mathbf{e_j})\overline{\mathbf{e_i}}.$$
Consequently,
 $$\Phi(Y_i)=\sum_{\overline{\mathbf{e_k}}\mathbf{e_j}=\pm \mathbf{e_i}}\mathbf{e_k}x_j(\overline{\mathbf{e_k}}\mathbf{e_j})\overline{\mathbf{e_i}}
 =\Big( \sum_{\overline{\mathbf{e_k}}\mathbf{e_j}=\pm \mathbf{e_i}}
 x_j  \mathbf{e_j} \Big)\overline{\mathbf{e_i}}=\Phi(X)\overline{\mathbf{e_i}}.$$
  This completes the proof.
\end{proof}

Now we can express any element $X_i\in \R^8$ in terms the  Hermitian variables,  and vise versa.

\begin{theorem}\label{thm:310} Let    $\{X_i\}_{i=0}^7$ be the twistor vectors in $\mathbb R^8$ and
 $\{Z_i\}_{i=0}^7$
 the  Hermitian   variables in    $\O\otimes C\ell_8$. For any $i=0,1,\ldots,7$, we have

  \begin{eqnarray}
  X_i&=&\frac{1}{8}\mathbf{e_i}\sum_{j=0}^7(-1)^{\sigma(j,i)} Z_j; \label{basis-102}
  \\ \notag \\
  Z_i&=&\sum_{j=0}^7J_i(\mathbf{e_j})X_j. \label{basis-134}
  \end{eqnarray}

\end{theorem}
\begin{proof}
It follows Proposition \ref{prop3.5} and Lemma \ref{lem-inv-235} that \eqref{basis-102} holds.
The converse identities follow from  Proposition \ref{prop3.5}.
\end{proof}

\begin{remark}
Using \eqref{eq:basis-589}, we can provide the explicit expression for the Hermitian variable$Z_i$:
  \begin{equation*}
    \begin{dcases}
    Z_0  =\mathbf{e_0}X_0+\mathbf{e_1}X_1+\mathbf{e_2}X_2+\mathbf{e_3}X_3+\mathbf{e_4}X_4+\mathbf{e_5}X_5+\mathbf{e_6}X_6+\mathbf{e_7}X_7; \\
    Z_1  =\mathbf{e_0}X_0-\mathbf{e_1}X_1+\mathbf{e_2}X_2-\mathbf{e_3}X_3+\mathbf{e_4}X_4-\mathbf{e_5}X_5+\mathbf{e_6}X_6-\mathbf{e_7}X_7; \\
    Z_2  =\mathbf{e_0}X_0+\mathbf{e_1}X_1-\mathbf{e_2}X_2-\mathbf{e_3}X_3+\mathbf{e_4}X_4+\mathbf{e_5}X_5-\mathbf{e_6}X_6-\mathbf{e_7}X_7; \\
    Z_3  =\mathbf{e_0}X_0-\mathbf{e_1}X_1-\mathbf{e_2}X_2+\mathbf{e_3}X_3+\mathbf{e_4}X_4-\mathbf{e_5}X_5-\mathbf{e_6}X_6+\mathbf{e_7}X_7; \\
    Z_4  =\mathbf{e_0}X_0+\mathbf{e_1}X_1+\mathbf{e_2}X_2+\mathbf{e_3}X_3-\mathbf{e_4}X_4-\mathbf{e_5}X_5-\mathbf{e_6}X_6-\mathbf{e_7}X_7; \\
    Z_5  =\mathbf{e_0}X_0-\mathbf{e_1}X_1+\mathbf{e_2}X_2-\mathbf{e_3}X_3-\mathbf{e_4}X_4+\mathbf{e_5}X_5-\mathbf{e_6}X_6+\mathbf{e_7}X_7; \\
    Z_6  =\mathbf{e_0}X_0+\mathbf{e_1}X_1-\mathbf{e_2}X_2-\mathbf{e_3}X_3-\mathbf{e_4}X_4-\mathbf{e_5}X_5+\mathbf{e_6}X_6+\mathbf{e_7}X_7; \\
    Z_7  =\mathbf{e_0}X_0-\mathbf{e_1}X_1-\mathbf{e_2}X_2+\mathbf{e_3}X_3-\mathbf{e_4}X_4+\mathbf{e_5}X_5+\mathbf{e_6}X_6-\mathbf{e_7}X_7.
      \end{dcases}
  \end{equation*}
\end{remark}

\begin{remark}
Here, we interpret Theorem \ref{thm:310}. In the context of complex analysis, we consider the function $$f:\C\to \C,$$
using two different systems: the real coordinate system $x,y$ and the complex coordinate system $z,\bar{z}$
which are respectively suitable for different environments.
In the case of Hermitian analysis, we consider the function $$f:\R^8\to \O\otimes C\ell_8$$
and two coordinate systems $\{X_i\}_{i=0}^7$ and $\{Z_i\}_{i=0}^7$.
The Hermitian variables correspond to the complex coordinate system in complex analysis,
thus Hermitian analysis corresponds to complex analysis.
\end{remark}
At last of this section, we point out that the twistor vectors   $X_i$ are orthogonal.
  $$(X_i,X_j)=\abs{X}^2\delta_{ij}.$$
Moreover,  the twistor vectors   $X_i$ generate the Clifford algebra $C\ell_{8}$.

  \begin{lemma}\label{eq:cliffod-twistor} For any $i,j=0,1,\ldots,7$, we have
\begin{equation}\label{X}
  X_iX_j+X_jX_i =-2\abs{X}^2\delta_{ij}
\end{equation}
\end{lemma}

  \begin{proof}
Indeed, we have
\begin{eqnarray*}
(X_i,X_j)&=& (\Phi(X_i),\Phi(X_j))
\\
&=&(\Phi(X)\overline{e_i},\Phi(X)\overline{e_j})
\\ &=& (\overline{e_i},\overline{\Phi(X)}(\Phi(X)\overline{e_j})) \\
             &=& (\overline{e_i},\abs{\Phi(X)}^2\overline{e_j})
             \\&=&\abs{\Phi(X)}^2\delta_{ij}=\abs{X}^2\delta_{ij}.
\end{eqnarray*}
\end{proof}

\section{Octonionic Hermitian Dirac Operators}
In this section, we define the octonionic Hermitian vector derivatives and octonionic Dirac operators.

For any vector $$X = \sum_{i=0}^7 x_i g_i \in \R^8,$$ its twistor vectors are defined by
$$X_i = \Phi^{-1}(\Phi(X)\overline{\mathbf{e_i}})$$
for $i = 0, \ldots, 7$.

We then define their Fischer dual operators as
$$\partial_{X_i} = \Phi^{-1}(\Phi(\partial_{X})\overline{\mathbf{e_i}}),$$
where the derivative with respect to $X$ is given by
$$\partial_{X} = \sum_{i=0}^7 g_i \partial_{x_i}.$$

\begin{definition}
	The octonionic Hermitian vector derivatives are defined as
	$$\partial_{Z_i} = f_i J_i(\Phi(\partial_{X}))$$
	for $i = 0, \ldots, 7$.
\end{definition}

From Theorem \ref{basis-134} and Lemma \ref{eq:cliffod-twistor}, it follows that
$$\partial_{Z_i} = \sum_{j=0}^7 J_i(\mathbf{e_j})\partial_{X_j},$$
and
\begin{align*}
	\partial_{X_i}\partial_{X_j} + \partial_{X_j}\partial_{X_i} = &-2\Delta\delta_{ij}
\end{align*}
where $$\Delta = \sum_{i=0}^7 \frac{\partial^2}{\partial x_i^2}$$ is the Laplacian operator on $\R^8$.

\section{Witt Bases in $\O\otimes C\ell_{8n}$ and $\O^n\otimes C\ell_{8n}$}
With the help of a finite subgroup in the automorphism of octonion,
we have constructed the Witt basis of  $\O\otimes C\ell_8$.
The same approach makes us   extend the result to the   tensor products  $$\O \otimes C\ell_{8n}, \qquad \O^n \otimes C\ell_{8n},$$
respectively.
\subsection{$\O \otimes C\ell_{8n}$}
First, we consider the space $\O \otimes C\ell_{8n}$. Let $\{g_0, g_1, \ldots, g_{8n-1}\}$ be the orthonormal basis of $\R^{8n}$, and define
$$\Omega_k = \sum_{i=0}^7 \overline{\mathbf{e_i}}g_{8k+i} \in \O\otimes C\ell_{8n}$$
for $k = 0, \ldots, n-1$.

\begin{definition}
	The octonionic Witt basis of $\O\otimes C\ell_{8n}$ is given by
	\begin{equation*}
		\begin{dcases}
			f_0^k &= J_0(\Omega_k), \\
			f_1^k &= J_1(\Omega_k), \\
			&\vdots \\
			f_7^k &= J_7(\Omega_k),
		\end{dcases}
	\end{equation*}
	for $k = 0, \ldots, n-1$, where $\Omega_k$ is as defined above.
\end{definition}

We start  from a generic element $$X=\sum_{i=0}^{8n-1}x_ig_i\in \R^{8n}.$$
There associate $n$-elements
$$X^{j}=\sum_{i=0}^7x_{8j+i}g_{8j+i}$$
for any $j=0,\ldots,n-1$.

Moreover, we have the  Witt decomposition
$$X^j=\frac{1}{8}\sum_{i=0}^7Z_i^j,$$
where
$$ Z_i^j:= f_i^jJ_i(\Phi(X^j)).$$
Consequently,  we  get $$X=\frac{1}{8}\sum_{j=0}^{n-1}\sum_{i=0}^7f_i^jJ_i(\Phi(X^j)).$$
Thus all the processes  in the last section work  in this general case.

\subsection{$\O^n \otimes C\ell_{8n}$}
 We next consider the space $\O^n \otimes C\ell_{8n}$.

 For any  $k=0,\ldots,n$ and $i=0,\ldots,7$, we denote $$\mathbf{e^k_i}=(0,\ldots,0,\mathbf{e_i},0,\ldots,0)\in \O^n,$$ where   $\mathbf{e_i}$ is at the   $k$-th slot.
Let $\{g_0,g_1,\ldots,g_{8n-1}\}$ be the orthonormal basis of $\R^{8n}$ as before.

 We consider
$$\hat{\Omega}_k=\sum_{i=0}^7\overline{\mathbf{e^k_i}}g_{8k+i} \in \O^n\otimes C\ell_{8n} $$
for any $k=0,\ldots,n$.

And we define a set of embedding maps
$$\Phi_k:\R^8\to \O^8, \qquad \forall\ k=0, 1, \ldots, n-1,$$ defined  by
$$\Phi_k\left(\sum_{i=0}^7x_ig_i\right)=\sum_{i=0}^7x_i\mathbf{e^k_i}.$$

\begin{definition}
	The octonionic Witt basis of $\O^n \otimes C\ell_{8n}$ is established as
	\begin{equation}
		\begin{dcases}
			f_0^k &= J_0(\hat{\Omega}_k), \\
			f_1^k &= J_1(\hat{\Omega}_k), \\
			&\vdots \\
			f_7^k &= J_7(\hat{\Omega}_k),
		\end{dcases}
	\end{equation}
	for $k = 0, \ldots, n-1$, where $\hat{\Omega}_k$ is defined as above.
	
	In the same way, our start point is a generic element $$X=\sum_{i=0}^{8n-1}x_ig_i\in \R^{8n}.$$
There associate $n$-elements
$$X^{j}=\sum_{i=0}^7x_{8j+i}g_{8j+i}$$
for any $j=0,\ldots,n-1$.

Moreover, we have the  Witt decomposition
$$X^j=\frac{1}{8}\sum_{i=0}^7Z_i^j,$$
where
$$ Z_i^j:= f_i^jJ_i(\Phi_i(X^j)).$$
Consequently,  we  get $$X=\frac{1}{8}\sum_{j=0}^{n-1}\sum_{i=0}^7f_i^jJ_i(\Phi(X^j)).$$
Thus all the processes  in the last section work  in this general case.
\end{definition}

Given the comprehensive exploration of the Witt basis in the context of octonion and Clifford algebras, the innovative approaches presented in this article pave the way for new mathematical inquiries and applications. The intricate relationship between octonionic automorphisms and the finite $G_2$-subgroup, as elucidated through binary expressions, opens up fertile ground for further research. This work not only extends the existing mathematical framework but also enhances our understanding of the structural complexities within higher-dimensional algebras. The establishment of the octonionic Witt basis and its application to Hermitian Dirac operators significantly broaden the scope of octonionic Hermitian analysis, marking a milestone in the field and encouraging the pursuit of additional groundbreaking studies.

\section*{Conflict of interest statement}
We declare that we have no conflict of interest.

\end{document}